%% file: picardnumbers2.tex
\documentclass[11pt,a4paper]{amsart}
\usepackage[utf8]{inputenc}
\usepackage[english]{babel}
\usepackage[dvipsnames]{xcolor}
\usepackage{tikz}
\usepackage{soul,color}

\newcommand*{\fullref}[1]{\hyperref[{#1}]{\ref*{#1}. \nameref*{#1}}} 

\usepackage[osf]{mathpazo}
\usepackage{charter}
\usepackage[a4paper, top=3.2cm, bottom=3.2cm,left=3.2cm, right=3.2cm, heightrounded,bindingoffset=0mm]{geometry}

\usepackage[colorlinks=true]{hyperref}
\hypersetup{bookmarksnumbered=true, linkcolor=due, citecolor=Green, urlcolor=black,}

\include{pack}
\begin{document}
	\title[On the Picard numbers of abelian varieties (char $p>0$)]{On the Picard numbers of abelian varieties\\in positive characteristic}
	\author{Roberto Laface}
	\dedicatory{Dedicated to my mother on the occasion of her 50th birthday}

\clearpage\maketitle
\thispagestyle{empty}
\begin{abstract}
In this paper, we study the set $R_g^{(p)}$ of possible Picard numbers of abelian varieties of dimension $g$ over algebraically closed fields of characteristic $p>0$. We show that many of the results for complex abelian varieties have analogues in positive characteristic: non-completeness in dimension $g \geq 2$, asymptotic completeness as $g \ra +\infty$, structure results for abelian varieties of large Picard number. On the way, we highlight and discuss new characteristic $p>0$ features and pathologies: non-additivity of the range of Picard numbers, supersingularity index of an abelian variety, dependence of $R_g^{(p)}$ on $p$, relation to the $p$-rank and the Newton polygon.
\end{abstract}

\setcounter{tocdepth}{1}

\section{Introduction}

Let $g \geq 1$ be an integer. In this note we would like to study the set $R_g^{(p)}$ of possible Picard numbers of abelian varieties of dimension $g$ over algebraically closed fields of characteristic $p > 0$. This is a continuation of the earlier work \cite{hulek-laface17}. 

For a smooth projective variety $X$ over an algebraically closed field $K$, the \textit{N\'eron-Severi group} $\NS(X)$ of $X$ is the set of divisors modulo algebraic equivalence. By Severi's Theorem of the base, it is known that $\NS(X)$ is a finitely generated abelian group, and its rank $\rho(X) := \rank \NS(X)$ is called the \textit{Picard number} of $X$. The Picard number is notably one of the most difficult invariants of an algebraic variety to calculate. In order to compute it, one should find enough codimension-one cycles to get a strong lower bound on $\rho(X)$, and at the same time look for an upper bound on $\rho(X)$. In some special cases, this can be achieved by finding automorphisms or special subvarieties (e.g.~lines on surfaces) on $X$.

In characteristic zero, the Lefschetz Theorem on (1,1)-classes implies that the Picard number of $X$ satisfies the inequality $1 \leq \rho(X) \leq h^{1,1}(X)$. For many classes of surfaces all such values actually occur: for instance, all integers $\rho$ such that $1 \leq \rho \leq 20$ (respectively, $1 \leq \rho \leq 4$) do occur as the Picard number of some (projective) K3 surface (respectively, abelian surface). 

Moving to positive characteristic instead, one sees that jumping phenomena for the Picard number are known to occur already in the case of surfaces. In this respect, K3 surfaces again provide an example: by a result of Igusa \cite{igusa60}, the Picard number of an algebraic surface $X$ over an algebraically closed field of characteristic $p>0$ satisfies the inequality $\rho(X) \leq b_2(X)$. In particular, for a K3 surface $X$, we have that $\rho(X) \leq 22$. However, one can never have $\rho(X) =21$, as remarked by Artin \cite{artin74}. 

In fact, as we will later show in detail, jumping phenomena occur also for abelian surfaces. This latter observation was the leading motivation behind our note. Our aim is to study the possible Picard numbers of abelian varieties of a fixed dimension $g$.\newline

If $X$ is an abelian variety of dimension $g$ and $\ell$ is a prime different from the characteristic of the base field, it is well-known that the \'etale cohomology groups $\rHet{n}(X,\Z_\ell)$ are determined by $\rHet{1}(X,\Z_\ell)$ by taking exterior powers. In particular, we have that $b_2(X) = {2g \choose 2} = 2g^2 - g$, from which we conclude that $\rho(X) \leq 2g^2-g$. We are then interested in studying which integers $\rho$ such that $1 \leq \rho \leq 2g^2-g$ occur as the Picard number of an abelian variety of dimension $g$.

Our strategy closely follows that of \cite{hulek-laface17}. As the Picard number of abelian varieties is invariant under isogenies, we can use the Poincar\'e Complete Reducibility Theorem \cite[IV, \S{19}, Theorem 1]{mumford08} to choose a suitable representative in the isogeny class. Then, as the Picard number is an additive function (but not strongly additive) \cite[IV, \S{21}, Theorem 6]{mumford08}, a result of Murty \cite[Lemma 3.3]{murty84} (which is independent of the characteristic of the base field) allows us to give an algorithm to compute the Picard number of any abelian variety.

This observation paves the way towards the study of the sets $R_g^{(p)}$, about which very little is known. Many of the results obtained for complex abelian varieties in \cite{hulek-laface17} carry over to positive characteristic, but there are also several important new features.\newline

As one might expect, supersingular abelian varieties play a major role in this new setting. This is motivated by the fact that supersingular elliptic curves belong to a unique isogeny class, and it allows us to define a new invariant, the supersingularity index, which will be one of our fundamental tools and will be closely related to the Picard number.

As over the complex numbers, there exist gaps in the set of possible Picard numbers of abelian varieties of dimension $g$. However, in positive characteristic, this is already true for $g=2$, while over $\C$ this phenomenon starts occurring in dimension $g \geq 3$ \cite{hulek-laface17}. After the bulk of the present work had been written, we were informed by Matthias Sch\"utt of a paper of Shioda \cite{shioda75}, in which he notices the existence of gaps for the Picard numbers of complex abelian threefolds and of abelian surfaces in positive characteristic, see \cite[Appendix]{shioda75}.

Another surprising difference with the characteristic zero case is the fact that the range of Picard numbers is not additive, that is there is no analogue of \cite[Proposition 6.1]{hulek-laface17} in any positive characteristic. This is intimately connected to the existence of supersingular abelian varieties, and we will later exhibit examples of this phenomenon.

Among the non-results, we have to mention the combinatorial structure of $R_g^{(p)}$ (over $\C$, this is \cite[Section 7.2]{hulek-laface17}). The main obstruction to such a result is the lack of existence results for abelian varieties with given endomorphism algebra. Also, the existence of abelian varieties with certain endomorphism algebras depends on the characteristic of the base field. We hope to pursue results in this direction in the future.

Nevertheless, we were able to recover asymptotic completeness for the Picard numbers of abelian varieties, and results on the structure of abelian varieties with large Picard number, which in positive characteristic have consequences on the $p$-rank and the Newton polygon.

	\begin{ack}
		I would like to warmly thank Bert van Geemen, Klaus Hulek, Christian Liedtke, Ben Moonen and Rachel Pries for their comments on the manuscript, and Matthias Sch\"utt for his comments and especially for pointing out Shioda's paper \cite{shioda75}. Special thanks go to Oliver Gregory for many fruitful discussions and detailed suggestions on how to improve earlier drafts of this paper. This research has been funded by the ERC Consolidator Grant 681838 K3CRYSTAL.
	\end{ack}


\section{An algorithm to compute the Picard number}

\subsection{Isogeny decomposition}
Let $p$ be a prime number and $g$ a positive integer. We are interested in studying the set $R_g^{(p)}$ of Picard numbers of abelian varieties of dimension $g$ in characteristic $p$, namely
\[ R_g^{(p)} := \big\lbrace \rho(X) \, \big\vert \, \text{$X/K$ abelian variety}, \, \dim X = g, \, \operatorname{char} K = p \big\rbrace. \]
Here, the base field $K$ is allowed to vary among all algebraically closed fields of characteristic $p$. As we will work fixing the characteristic, we will use the simplified notation $R_g$ in place of $R_g^{(p)}$ whenever there is no risk of confusion.

Recall that if $E$ is an elliptic curve over $k$, the absolute Frobenius automorphism $F$ acts on the cohomology group $\rH^1(E, \ko_E)$ by pull-back. Since $h^1(\ko_E)=1$, the $F$-linear map $F^*$ is either zero or bijective. Whenever $F^*=0$, we say that $E$ is \textit{supersingular}. If now $X/k$ is an abelian variety, we say that $X$ is \textit{supersingular} if it is isogenous to a product $E_1 \times \cdots \times E_g$, where $g = \dim X$ and the $E_i$'s are supersingular elliptic curves (this definition is a consequence of \cite{oort74}). 

Any two supersingular elliptic curves are isogenous: this follows, for instance, from \cite[Ch.~13, Theorem 6.4]{husemoller04} together with \cite[Ch.~13, Theorem 8.4]{husemoller04}. In light of the above results, we choose a supersingular elliptic curve $E$, so that every supersingular abelian variety of dimension $g \geq 1$ is isogenous to $E^g$ (and we keep this choice for the remainder of this article). We will shortly see that the existence of a unique isogeny class of elliptic curves in characteristic $p$ has deep consequences for the type of results that we are seeking.
 
The Poincar\'e complete reducibility theorem \cite[Theorem 5.3.7]{birkenhake-lange04} we resorted in \cite{hulek-laface17} holds also for abelian varieties over arbitrary algebraically closed fields \cite[IV, {\S}19, Theorem 1]{mumford08}. Then, given an abelian variety $X$ of dimension $g$, we can decompose it in its isogeny class as 
	\[ X \sim X_1^{n_1} \times \cdots \times X_r^{n_r} \times E^s, \]
	where 
\begin{itemize}
	\item $\Hom(X_i,E) = \Hom(X_i,X_j)=0$ for all $(i,j)$ with $i \neq j$;
	\item $s+\sum_{i=1}^r n_i \dim X_i = g$, $n_i>0$ for all $i$ and $s\geq 0$.
\end{itemize}

\begin{defi}
The non-negative integer $s \equiv s(X)$ is called the \textit{supersingularity index} of $X$. Equivalently, it is the dimension of the largest supersingular abelian subvariety of $X$.
\end{defi}
Notice that $s=0$ corresponds to the situation where $X$ does not contain any supersingular abelian subvariety. In this case, the Picard number should morally behave as it does over $\C$. The supersingularity index is a new feature that makes sense in positive characteristic only, and we will later study its interplay with the Picard number.

\subsection{Reduction steps}
	For abelian varieties, the Picard number is an isogeny invariant, hence we can replace a given abelian variety by the representative of its isogeny class given in the Poincar\'e Complete Reducibility Theorem \cite[IV, \S{19}, Theorem 1]{mumford08}. Indeed, letting $\varphi: X \longra Y$ be an isogeny of degree $d$, there exists an isogeny $\psi:Y \longra X$ such that $\psi \circ \varphi = [d]_X$ and $\varphi \circ \psi = [d]_Y$. In particular, one has a sequence of morphisms
\[ \NS(A) \xrightarrow{\psi^*} \NS(B) \xrightarrow{\varphi^*} \NS(A),\]
whose composition is $[d]^*_X$. As $[d]^*_X$ induces multiplication by $d^2$ on $\NS(A)$, we see that $\rho(A) \leq \rho(B)$, and the opposite inclusion is obtained by exchanging the role of $A$ and $B$. 
	
	Moreover, in \cite[Corollary 2.3]{hulek-laface17}, we observed that the Picard number of abelian varieties is additive (but not strongly additive, see Proposition \ref{murty}). The proof given therein only works in the complex setting. An account of this in full generality may be found (although implicit) in \cite[IV, \S{21}, Theorem 6]{mumford08}. However, it is not hard to come up with a proof valid over arbitrary fields of arbitrary characteristic by means of the Poincar\'e line bundle. In conclusion, we have the following:
	
	\begin{prop}[Additivity of the Picard number]\label{splitting}
		Let $X_1, \dots, X_r$ be non-pairwise isogenous simple abelian varieties over $k$, and let $n_1, \dots, n_r$ be positive integers. Then,
		\[ \rho\Bigg( \prod_{i=1}^r X_i^{n_i} \Bigg) = \sum_{i=1}^r \rho(X_i^{n_i}). \]
	\end{prop}

	Thus, in order to compute $\rho(X)$, we are left with computing $\rho(A^k)$ for $A$ a simple abelian variety and $k \geq 1$. In \cite[Lemma 3.3]{murty84}, Murty states without giving a proof such a result for complex abelian varieties. In \cite{hulek-laface17}, the authors applied Murty's result to compute 
	the Picard number of a self-product of a complex abelian variety. In fact, it turns out that this result is independent of the complex setting\footnote{This fact is implicit in \cite[IV, \S{21}]{mumford08}.}. 
	
	Before stating the result, let us briefly recall the notation, and refer the reader to \cite{albert34a}, \cite{albert34}, \cite{albert35} for further details, or to \cite{birkenhake-lange04} for a modern account. If $(D,*)$ is an Albert algebra, let $K =Z(D)$ be its center, which is a number field. Set 
	\[e:=[K : \Q], \qquad d^2 :=[D:K] \qquad  e_0:=[K_0:\Q],\]
	where $K_0$ is the maximal totally real subfield of $K$. Then, the endomorphism algebras of simple abelian varieties are divided into four types:
	
	\begin{itemize}
		\item Type I$(e)$: $d=1$, $e=e_0$ and $D = K = K_0$ is a totally real field;
		\item Type II$(e)$: $d=2$, $e=e_0$ and $D$ is a totally indefinite quaternion algebra over the totally real field $K=K_0$;
		\item Type III$(e)$: $d=2$, $e=e_0$ and $D$ is a totally definite quaternion algebra over the totally real field $K=K_0$;
		\item Type IV$(e_0,d)$: $e=2e_0$ and $D$ is an algebra over the CM field $K \supset K_0$.
	\end{itemize}
	
	\begin{prop}[Lemma 3.3 of \cite{murty84}]\label{murty}
		Let $A$ be a simple abelian variety, and let $D=\End(A) \otimes \Q$. Then, for $k \geq 1$, one has
		\[
		\rho(A^k) = 
		\begin{cases} 
		\frac{1}{2}ek(k+1) & \text{\rm{Type I$(e)$}}\\
		ek(2k+1) & \text{\rm{Type II$(e)$}}\\
		ek(2k-1) & \text{\rm{Type III$(e)$}}\\
		e_0d^2k^2 & \text{\rm{Type IV$(e_0,d)$}}
		\end{cases} \quad = \quad
		\begin{cases} 
		\frac{1}{2}\rho(A)k(k+1) & \text{\rm{Type I$(e)$}}\\
		\frac{1}{3}\rho(A)k(2k+1) & \text{\rm{Type II$(e)$}}\\
		\rho(A)k(2k-1) & \text{\rm{Type III$(e)$}}\\
		\rho(A)k^2 & \text{\rm{Type IV$(e_0,d)$}}
		\end{cases}	.
		\]
	\end{prop}

Murty does not give any proof of this fact, and thus we would like to give a hint on how this works. 

\begin{proof}[Sketch of the proof]
By \cite[IV, {\S}19, Corollary 2]{mumford08} we have that
\[ \End(A^k)\otimes \Q \cong M_{k}(D),\]
where $D := \End(A) \otimes \Q$ is a division algebra. Let us now suppose that $A$ is simple of Type I, that is $D = K = K_0$ is a totally real number field of degree $e$ over $\Q$, and the Rosati involution is trivial on $D$. By tensoring with $\R$, 
\[ \End(A^k) \otimes \R = M_k(D) \otimes \R \cong M_k(D \otimes \R) \cong M_k(\R)^{\oplus e}. \]
By the proof of \cite[IV, \S{21}, Theorem 6]{mumford08}, the Rosati involution $\dagger$ acts separately on each of the factors $M_k(\R)$, and we can also assume that it is given by $X \mapsto \bar{X}^t$ on each factor. Then,
\[ \NS(A^k) \otimes \R \cong \big( M_k(\R)^{\oplus e} \big)^\dagger \cong \big( M_k(\R)^{\dagger} \big)^{\oplus e} \cong \kh_k(\R)^{\oplus e},\]
where, following the notation in \cite{mumford08}, $\kh_k(\R)$ denotes the space of $k \times k$ symmetric matrices with real entries. It follows that $\rho(A^k) = \frac{1}{2}ek(k+1)$. For Type II, the computation is completely analogous, while for Type III (respectively, IV) we see that the factors of $\NS(A^k)$ are of the form $\kh_k(\IH)^{\oplus e}$, the space of Hermitian quaternionic matrices (respectively, $\kh_k(\C)^{\oplus e_0}$, the space of complex Hermitian matrices).
\end{proof}


\section{Bounds on the Picard number}

\subsection{Bounds for self-products of abelian varieties}\label{restriction}
We will now discuss how to bound the Picard number of an abelian variety $X$ isogenous to a self-product $A^k$, $A$ being a simple abelian variety. Already when $k=1$, the situation is different from the one in characteristic zero. In fact, over fields of characteristic $p >0$, the usual necessary conditions on the Picard number over $\C$ (see, for instance, \cite[Proposition 5.5.7]{birkenhake-lange04}) are replaced by the sensibly weaker following restrictions:

\begin{itemize}
	\item for Type I$(e)$, $e \vert g$;
	\item for Type II$(e)$, $2e \vert g$;
	\item for Type III$(e)$, $e \vert g$;
	\item for Type IV$(e_0,d)$, $e_0 d \vert g$.
\end{itemize}

	For a survey on these conditions, see \cite{oort88}. We will come back to these restrictions later on (see Section \ref{pathologies}), but for now we will use them to bound the Picard number of $A^k$, independently of the endomorphism type of $A$. 
	
\begin{prop}\label{self-product}
		Let $X$ be an abelian variety of dimension $g$ that is isogenous to a self-product of a simple abelian variety. Then,
		\[ \rho(X) \leq g^2, \]
		unless $X$ is supersingular, in which case $\rho(X) = 2g^2-g$.
\end{prop}

\begin{proof}
		We can assume that $X= A^k$, with $A$ a simple abelian variety of dimension $n$ (so that $g=nk$). If $A=E$ is an elliptic curve, then $X= E^g$ and
		\[ \rho(E^g) = 
		\begin{cases}
		\frac{1}{2} g(g+1) & \text{$\End(E) \otimes \Q \cong \Q$},\\
		g^2 & \text{$\End(E) \otimes \Q \cong \Q(\sqrt{-d})$ \ (for some $d>0$)},\\
		2g^2 - g & \text{$E$ is supersingular.}
		\end{cases}
		\]
		Otherwise, $n \geq 2$, and $\rho(A^k) \leq g^2$ by a Type-by-Type analysis as in \cite{hulek-laface17}.
\end{proof}

\subsection{Bounding the Picard number by length}

	In \cite[Theorem 1.1]{hulek-laface17}, we showed that the range of Picard numbers of abelian varieties of dimension $g$ exhibits gaps already for $g \geq 3$. In this section, we would like to show that an analogous result holds in positive characteristic.
	
	In order to do so, we will introduce the {length} of (an isogeny decomposition of) a given abelian variety. If $X$ is an abelian variety and 
	\[ X \sim X_1^{n_1} \times \cdots \times X_{r}^{n_{r}} \]
	is an isogeny decomposition according to the Poincar\'e Complete Reducibilty Theorem, then we set $r(X):=r$ and we call $r(X)$ the \textit{length} of $X$. This quantity is certainly well-defined, as the factors and the powers appearing in the decomposition are determined up to isogenies and permutations.
	
	For $r \leq g$, we define the following positive integer:
	\[ M_{r,g} := \max \lbrace \rho(X) \, \vert \, \dim X = g, \ r(X) = r \rbrace. \]
	This is the maximum Picard number that can be realized by using an abelian variety of dimension $g$ and length $r$. Our first result concerns the value of $M_{r,g}$.
	
\begin{prop}\label{bound_length}
		For positive integers $r,g \in \IN$ such that $r \leq g$, one has that 
		\[M_{r,g}=  \big[ 2(g-r+1)^2 - (g-r+1) \big]+(r-1).\]
	This value is obtained as the Picard number of $E^{g-r-1} \times E_1 \times \cdots \times E_{r-1}$, where $E$ is supersingular, and the elliptic curves appearing in the product are not mutually pairwise isogenous.	
\end{prop}

	\begin{proof}
		Suppose $X$ is an abelian variety of dimension $g$ and length $r$. We must distinguish two cases, according to whether the abelian variety we consider contains supersingular abelian subvarieties. If $X$ does not contain any supersingular abelian subvariety, then
		\[ X \sim X_1^{n_1} \times \cdots \times X_r^{n_r} \quad \text{and} \quad \rho(X) \leq \sum_{i=1}^r \big( n_i \cdot \dim X_i \big)^2.\]
		By \cite[Proposition 3.1]{hulek-laface17}, $\rho(X) \leq [g-(r-1)]^2+(r-1)$. Therefore, we can assume that $X$ contains a supersingular abelian subvariety. In this case, for a fixed supersingular elliptic curve $E$, one has the isogeny decomposition
		\[ X \sim X_1^{n_1} \times \cdots \times X_{r-1}^{n_{r-1}} \times E^s. \]
		By Proposition \ref{splitting} and Proposition \ref{self-product}, $\rho(X) \leq \sum_{i=1}^{r-1} \big( n_i \dim X_i \big)^2 + (2s^2-s)$. Now, consider the function 
		\[ f(\underline{x}) := x_1^2 + \cdots + x_{r-1}^2 + (2x_r^2 - x_r) \]
		in the domain $\Omega := \lbrace \underline{x}=(x_1, \dots, x_r) \in \R^r {\,} \vert {\,} x_1 + \cdots + x_r = g, \ x_i \geq 1 \ \forall i \rbrace$. $M_{r,g}$ is bounded from above by $M:=\max_{\underline{x} \in \Omega} \lbrace f(\underline{x}) \rbrace$, which we now compute.
		
		For this, just notice that $(y-f(\underline{x})=0)$ defines a paraboloid in $\R^{r+1}$, and that the maximum of $f$ on $\Omega$ is attained at the point of $\Omega$ that is furthest from the vertex of the paraboloid. As $f$ is symmetric in $x_1, \dots, x_r$ we only need to check the values of $f$ at the points $(1,\dots, 1, g-r+1)$ and $(g-r+1, 1, \dots, 1)$. It follows that 
		\[M = \big[ 2 (g-r+1)^2 - (g-r+1) \big]+(r-1).\]
		Finally, we can easily compute that if $Z:= E^{g-r-1} \times E_1 \times \cdots \times E_{r-1}$, where $E$ is supersingular and the elliptic curves appearing in the product are not mutually pairwise isogenous, then $\rho(Z) = M$. This shows that $M_{r,g} = M$, and we are done.
	\end{proof}
	
	\begin{cor}
		Let $X$ be an abelian variety of dimension $g$ and length $r$. Then,
		\[ \rho(X) = M_{r,g} \Longleftrightarrow X \sim E^{g-r-1} \times E_1 \times \cdots \times E_{r-1},\]
		where $E$ is supersingular and $\Hom(E,E_i) = \Hom(E_i,E_j)=0$ for all $i \neq j$.	
	\end{cor}

As a concluding remark, let us notice that the sequence of integers
		\[ g= M_{g,g} < M_{g-1,g} < \cdots < M_{3,g} < M_{2,g} < M_{1,g} = 2g^2 -g \]
		is strictly increasing.
		

\section{Existence of gaps and a structure result}

\subsection{Gaps in the sequence of Picard numbers}\label{gaps}
We will now show that, starting from dimension $g \geq 2$, there exist gaps in the set of attainable Picard numbers in positive characteristic. This yields yet another difference with the characteristic zero case, as all integers $\rho$ with $1 \leq \rho \leq 4$ occur as the Picard number of some abelian surface over $\C$. Our aim, in fact, is to prove the following result, which shows the existence of two precise gaps as the dimension is large enough, and computes the three largest Picard numbers for abelian varieties in characteristic $p >0$ for $g \geq 7$.

	\begin{thm}\label{first_gaps}
		Let $X$ be an abelian variety of dimension $g$.
		\begin{enumerate}
			\item If $g \geq 5$, then $\rho(X) \notin \Big( \big[2(g-1)^2 -(g-1) \big] +1 , 2g^2-g \Big)$.
			\item If $g \geq 7$, then $\rho(X) \notin \Big( \big[2(g-2)^2 -(g-2) \big] +4 , \big[2(g-1)^2 -(g-1) \big] +1 \Big)$. 
		\end{enumerate}
	\end{thm}
	
Before proving this result, let us see what happens in the low dimensional cases. We urge the reader to remember that we cannot use additivity of the Picard number (as in \cite[Proposition 6.1]{hulek-laface17}).

\begin{ex}
	If $X$ has dimension $g=2$, then Proposition \ref{splitting} and Proposition \ref{self-product} imply $5 \notin R_2$ (this had already been noticed by Shioda in \cite[Appendix]{shioda75}). It is also easy to check that $R_2 \supset \lbrace 2,3,4,6\rbrace$ by just using products of elliptic curves. This shows that $R_2 =\lbrace 1,2,3,4,6\rbrace$, and in particular that gaps in the set of Picard number over field of positive characteristic naturally appear already in dimension two. In characteristic zero, they first occur in dimension three, as shown in \cite{hulek-laface17}. 
\end{ex}

\begin{ex}
	In dimension $g=3$, it is straightforward to show that $R_3 \supseteq \lbrace 1, \dots, 7, 9 , 15 \rbrace$: indeed, one computes the Picard number of products of elliptic curves and products of abelian surfaces (i.e.~we use the knowledge of $R_2$) with a very general elliptic curve (which has Picard number one), together with the fact that a very general abelian variety has Picard number one. Now $k \notin R_3$ for $10 \leq k \leq 14$, as one readily checks by using Propositions \ref{splitting} and \ref{self-product}. On the other hand, one has to argue slightly differently to prove that $8 \notin R_3$: for this, notice that the only way of realizing $8\in R_3$ is to use a simple abelian variety of Type IV and dimension three. However, by using the restrictions in Subsection \ref{restriction}, it follows that this case does not occur. Hence $R_3 = \lbrace 1, \dots, 7,9,15 \rbrace$.
\end{ex}

\begin{ex}
	For $g=4$, one can similarly see that $R_4 \supseteq \lbrace 1, \dots, 10, 16, 28 \rbrace$. Also, the only way of realizing the Picard numbers $\rho \in \lbrace 11, \dots, 15 \rbrace$ is to use simple abelian varieties of Type IV and dimension four. However, the divisibility conditions make it impossible to have an abelian variety with such Picard numbers. Therefore, we conclude that $R_4 = \lbrace 1, \dots, 10, 16, 28 \rbrace$. We will come back to this example in Section \ref{pathologies}
\end{ex}

	\begin{proof}[Proof of Theorem \ref{first_gaps}]
		Let us consider the isogeny decomposition of $X \sim X_1^{n_1} \times \cdots \times X_r^{n_r}$. If $r(X) =1$, then $\rho(X) \leq g^2$ by Proposition \ref{self-product}, unless $X$ contains a supersingular abelian subvariety (in which case $\rho(X) = 2g^2-g$). Therefore, we can assume that $r(X) \geq 2$.
		
		In this case, let us write $X \sim A_1 \times A_2$, where $\Hom(A_1, A_2)=0$, and let us set $n:= \dim A_1$, so that $\dim A_2 = g-n$ and $1 \leq n \leq g$. Now, if $X$ does not contain any supersingular abelian subvariety, then $\rho(X) \leq (g-1)^2 +1$ by \cite[Theorem 1.1]{hulek-laface17}. Otherwise, we can suppose that $A_1 = E^n$, where $E$ is a supersingular elliptic curve. By Propositions \ref{splitting} and \ref{self-product}, $\rho(X) \leq [2n^2-n] + (g-n)^2$. Consider now the function of one real variable
		\[ f({x}) := [2x^2-x] + (g-x)^2, \]
		over $\Omega:=[1,g-1]$. Its maximum in $\Omega$ is $M= f(g-1) = \big[ 2(g-1)^2 - (g-1)\big] + 1$, and it is achieved by the abelian variety $E^{g-1} \times E'$, where $E$ is supersingular and $E'$ is not isogenous to $E$. Now, notice that $M = M_{2,g}$, so that, in order to obtain (1), we only need to impose the condition $g^2 \leq M$, which is satisfied for $g \geq 5$.
		
		As for (2), notice that if $r(X) \geq 3$, then $\rho(X) \leq M_{3,g}$. Therefore, we can assume that $r(X) \leq 2$. If $r(X) =1$, Proposition \ref{self-product} implies that $\rho(X) \leq g^2$, unless $X$ is supersingular. Hence, we are left with the case $r(X) =2$. In this case, if $X$ does not contain any supersingular abelian subvariety, the proof of \cite[Theorem 1.1]{hulek-laface17} shows that $\rho(X) \leq (g-2)^2 + 4$. Let us now assume that $X$ does contain a supersingular subvariety, that is, we can write $X$ in its isogeny class as $X \sim E^n \times A$, where $E$ is a supersingular elliptic curve. 
		
		We distinguish two cases, according to the isogeny decomposition of $A$. If $A \sim F^{g-n}$, $F$ being an elliptic curve not isogenous to $E$, then the proof of Part (1) shows that $\rho(X) \leq\ \big[2(g-1)^2 -(g-1) \big] +1$. Since the maximum value is achieved for $n =g-1$, let us now assume assume $1 \leq n \leq g-2$ and come up with a sharper bound. Under this condition, the largest possible Picard number is $\big[2(g-2)^2 -(g-2) \big] +4$, and it is achieved for $n=g-2$ and $F$ an elliptic curve with complex multiplication. Finally, if $A$ has dimension at least two, then $n \leq g-2$ and thus the same considerations of the previous case apply, and we are done.
	\end{proof}
	
\subsection{Structure of abelian varieties with large Picard number}

The reader may wonder whether a structure theorem for abelian varieties of large Picard number in the same fashion of \cite[Theorem 4.2]{hulek-laface17} holds in positive characteristic. By definition, supersingular abelian varieties in characteristic $p$ are isogenous to a self-product of a supersingular elliptic curve. Being supersingular is also equivalent to having Picard number equal to the second Betti number. This means that, for an abelian variety, having maximal Picard number forces the isogeny class to have a very specific structure. In fact, such a structure result also holds for abelian varieties having the second and third largest Picard number.

	\begin{thm}\label{large_picard}
		Let $X$ be an abelian variety of dimension $g$.
		\begin{enumerate}
			\item Suppose $g \geq 6$. Then,
			\[ \rho(X) = \big[2(g-1)^2 -(g-1) \big] +1 \Longleftrightarrow X \sim E^{g-1} \times F, \]
			where $E$ is supersingular and it is not isogenous to $F$. 
			\item Suppose $g \geq 8$. Then,
			\[ \rho(X) = \big[2(g-2)^2 -(g-2) \big] +4 \Longleftrightarrow X \sim E^{g-2} \times F^2, \]
			where $E$ is supersingular and $F$ is an elliptic curve such that $\End(F) \otimes \Q \cong \Q(\sqrt{-d})$, for some $d >0$. 
		\end{enumerate}
	\end{thm}

	\begin{proof}
		The "if" implications in (1) and (2) are immediate. Conversely, suppose that $\rho(X)= \big[2(g-1)^2 -(g-1) \big] +1$. As $\rho(X) = M_{2,g}$, we must have $r(X) \leq 2$. By Proposition \ref{bound_length}, in order to prove (1), it suffices to show that $r(X) =2$. If $r(X)=1$, then $\rho(X) \leq g^2$ ($X$ cannot be supersingular). Now, the condition $g \geq 6$ ensures that $g^2 < \big[2(g-1)^2 -(g-1) \big] +1$, yielding a contradiction. Therefore, $r(X)=2$ and (1) is proven.
		
		As for (2), if $\rho(X) = \big[2(g-2)^2 -(g-2) \big] +4$, then $\rho(X) > M_{3,g}$ and $r(X) \leq 2$. As in the previous case, $g \geq 8$ implies that $r(X) =2$. Now, if $X$ does not contain any supersingular abelian subvariety, then $\rho(X) \leq (g-2)^2 + 4$ by \cite[Theorem 4.2]{hulek-laface17}. Otherwise, we can write $X \sim E^n \times A$, $E$ being a supersingular elliptic curve and $A$ an abelian variety with $\Hom(E,A) =0$. If $A \sim F^{g-n}$, $F$ being a non-supersingular elliptic curve, then $\rho(X) \leq \big[2(g-1)^2 -(g-1) \big] +1$ and the maximum possible is attained for $n=g-1$, which contradicts the hypothesis on $\rho(X)$. Hence, we can assume that $1 \leq n \leq g-2$ and obtain the bound $\rho(X) \leq \big[2(g-2)^2 -(g-2) \big] +4$, which is attained only for $n = g-2$ and $F$ an elliptic curve with $\End(F) \otimes \Q \cong \Q(\sqrt{-d})$, for some $d >0$. Finally, if $A$ is a self-product of an abelian variety of dimension at least two, a similar argument shows that $\rho(X) \leq \big[2(g-2)^2 -(g-2) \big] +4$, and the maximum is attained only for $n=g-2$ and $A$ is a self-product of a simple abelian variety of CM-type (necessarily of dimension one or two). However, by \cite[Proposition 6.1]{oort88}, one sees that $A$ must be isogenous to a self-product of an elliptic curve $F$ with $\End(F) \otimes \Q \cong \Q(\sqrt{-d})$ ($d >0$), and we are done. 
		\end{proof}
		
		Although the above result recovers the picture in characteristic zero, the proof shows how abelian varieties of CM-type can now appear as fundamental blocks for the structure of abelian varieties of large Picard number. In fact, the proof of Theorem \ref{large_picard}(2) uses in a fundamental way a structure result for the endomorphism algebra of simple abelian surfaces \cite[Proposition 6.1]{oort88}, as opposed to the proof of \cite[Theorem 4.2]{hulek-laface17} which is purely combinatorial. For further remarks on abelian varieties of CM-type and their Picard number, we defer to Section \ref{pathologies}.
		

\section{Asymptotic results}\label{asymp_res}

\subsection{Asymptotic density of Picard numbers}

We have shown how the set $R_g$ is not complete already if $g \geq 2$, that is $\#R_g < 2g^2-g$. Instead, we will now prove that asymptotic density of Picard numbers holds in any positive characteristic. As in \cite{hulek-laface17}, the \textit{density of $R_g$} in $[1,2g^2-g]\cap \IN$ is the quantity $\delta_g := \#R_g / (2g^2-g)$, and, by taking the limit, the \textit{asymptotic density of $R_g$} is defined to be $\delta := \lim_{g \ra + \infty} \delta_g$. Our result is the following:

	\begin{thm}[Asymptotic completeness]
		Picard numbers of abelian varieties are asymptotically complete in any characteristic. In other words, $\delta =1$.
	\end{thm}
	
	\begin{proof}
		The proof follows the lines of \cite[Theorem 7.1]{hulek-laface17}: pick $n \in \IN$ such that $n \leq 2g^2 -g$, and let $n_1 \in \IN$ be the largest integer such that
		\[ 2n_1^2 - n_1 \leq n_1 < 2(n_1+1)^2 - (n_1 +1). \]
		From $2n_1^2 - n_1 \leq n-1$, it follows that $n_1 \leq \frac{1}{4}+\frac{1}{4}\sqrt{8n-7}$, which in turn implies that 
		\[0 \leq n-1-(2n_1^2 -n_1) \leq 4n_1 \leq 1 + \sqrt{8n-7} \leq 1 + \sqrt{16g^2 - 8g - 7}.\] 
		Set $m:= n-1 -(2n_1^2-n_1)$, and write $m$ as a sum of four square by Lagrange's Theorem on four squares: $m = n_2^2 + n_3^2 + n_4^2 + n_5^2$. We claim that $\sum_{i=1}^5 n_i < g$ for $g \gg 0$.
		
		By using the power mean, we find $n_2 + n_3 + n_4 + n_5 \leq 2\sqrt{1+\sqrt{16g^2 - 8g -7}}$, from which it follows that
		\[ \sum_{i=1}^5 n_i \leq \frac{1}{4}+\frac{1}{4}\sqrt{8n-7} + 2\sqrt{1+\sqrt{16g^2 - 8g -7}} \leq \sqrt{n/2} + 4 \sqrt{g+1}.\] 
		It can be checked that this the right-hand side is smaller than $g$ if and only if $n < 2g^2 - 16g\sqrt{g+1} + 32(g+1)$. Now, we are done by virtue of the next obvious result.
	\end{proof}
	
	\begin{lemma}
		Suppose $g \geq 1$ and $1 \leq n \leq 2g^2 - g$, where $g,n \in \IN$. Assume that there exist positive integers $n_1 , \dots, n_k$ such that 
		\[ n_1 = (2n_1^2 -n_1) + n_2^2 + \cdots + n_k^2 \qquad \text{and} \qquad n_1 + \cdots + n_k \leq g-1.\]
		Then, there exists a $g$-dimensional abelian variety $X$ with $\rho(X) = n$.
	\end{lemma}

\subsection{Asymptotic distribution of Picard numbers}\label{str_AV_large}

Let $R_g^*$ be the set of Picard numbers of abelian varieties of dimension $g$ and vanishing supersingularity index. We introduce a new subset of $R_g$, namely
\[ R_{g,n} := \big\lbrace [2(g-n)^2 - (g-n)]+x \, \vert \, x \in R_n^* \big\rbrace. \]
This is the subset of $R_g$ obtained by translating $R_n^*$ to the right. Notice that these are exactly the Picard numbers that can be attained by using products of the form $A_n \times E^{g-n}$, where $A_n$ has dimension $n$ and vanishing supersingularity index, and $E$ is a supersingular elliptic curve. The reader has certainly realized that, unlike in \cite[Subsection 7.2]{hulek-laface17}, we really cannot take $x \in R_n$, exactly because supersingular abelian varieties of a given dimension are all mutually isogenous. This will be explained in greater detail in Section \ref{pathologies}.

Analogously to \cite{hulek-laface17}, we will now define "large" Picard numbers. First of all, we will only concern ourselves with those Picard numbers $\rho > g^2$. Thus, the abelian varieties involved in our analysis have necessarily nonzero supersingularity index. Morally, an abelian variety with vanishing supersingularity index behaves as if it were defined over $\C$. Formally, this is a consequence of Propositions  \ref{splitting} and \ref{self-product}.

Secondly, by Theorem \ref{first_gaps}, we know that $R_{g,a} \cap R_{g,b} = \emptyset$, for small $a,b \in \IN$, and therefore we will only consider those set of the form $R_{g,n}$ that are mutually disjoint. More precisely, we are asking that the following two numerical conditions be satisfied:
	\begin{align}
		& g^2 < [2(g-n)^2 -(g-n)] +1;\\
		[2(g-n-1)^2 -(g-& n-1)] + (n+1) < [2(g-n)^2 -(g-n)]+1.
	\end{align}

For a fixed $g$, we say that a Picard number $\rho \in R_g$ is \textit{large} if there exists $n \in \IN$ such that $\rho \geq [2(g-n)^2 -(g-n)]+1$ and $n$ satisfies conditions (1) and (2). We can make this condition numerically precise.

	\begin{lemma}
		Let $g \geq 5$. Then, $\rho \in R_{g,n}$ is a large Picard number if and only if
		\[ n \leq \min \Bigg\lbrace \frac{4g-1-\sqrt{8g^2-7}}{4}, -3+\sqrt{4g+6} \Bigg\rbrace.\]
	\end{lemma}
	
In the next result, we study the distribution of large Picard numbers in the interval $[1,2g^2-g]$. As a byproduct, we deduce that there are more and more gaps in $R_g$ as $g \ra +\infty$. 

	\begin{thm}[Distribution of large Picard numbers]\label{distribution}
		For every positive integer $\ell$, there exists a genus $g_\ell$ such that for all $g \geq g_\ell$ one has that
		\[ \Big[ \big[2(g-\ell)^2-(g-\ell) \big] +1, 2g^2-g \Big] \cap R_g = R_{g,\ell} \sqcup R_{g,\ell -1} \sqcup \cdots \sqcup R_{g,2} \sqcup R_{g,1} \sqcup \lbrace 2g^2-g \rbrace. \]
		In other words, for all $g \geq g_\ell$, the Picard numbers in $R_g$ are distributed as follows:
		\[  \boxed{R_{g,\ell}} \qquad  \cdots  \qquad \boxed{R_{g,4}} \qquad \boxed{R_{g,3}} \qquad \boxed{R_{g,2}} \qquad \bullet^{[2(g-1)^2-(g-1)]+1} \qquad \bullet^{2g^2-g}.  \]
	\end{thm}
	
For sake of brevity, we will not report the proof of the theorem above, since it follows the lines of the proof of \cite[Theorem 7.4]{hulek-laface17}. 

\subsection{Structure of abelian varieties with large Picard number}
We will now remark that a striking consequence of Theorem \ref{distribution} is a description of the structure of abelian varieties of large Picard number. First, we can describe the isogeny classes of abelian varieties of large Picard number: the proof is analogous to the one of \cite[Corollary 7.6]{hulek-laface17}, and thus we omit it.

	\begin{cor}\label{s_vs_rho}
		For every positive integer $\ell$, there exists a genus $g_\ell$ such that for all $g \geq g_\ell$ and $n \leq \ell$ the following conditions are equivalent:
		\begin{enumerate}
			\item $\rho(X) \in R_{g,n}$;
			\item $s(X) = g-n$.
		\end{enumerate}
	\end{cor}

This result describes the isogeny structure of abelian varieties of large Picard number. Namely, under the assumption of the theorem, $\rho(X) \in R_{g,n}$ is equivalent to $X$ being isogenous to a product $E^{g-n} \times A_n$, where $A_n$ is an abelian variety of dimension $n$ with $\Hom(E,A_n) =0$ (in particular, $s(A_n) = 0$).

Moreover, the Picard number also affects the $p$-rank. In order to see this, recall that for an abelian variety $X/k$ of dimension $g$, $X[p](k) \cong ( \Z / p \Z )^f$ for some integer $f$ satisfying $0 \leq f \leq g$. The integer $f$ is called the \textit{$p$-rank} of $X$. If $X$ is supersingular, then $X[p](k) = 0$. However, outside of dimension one and two, the converse is false: for every $g \geq 3$, there exists an abelian variety $X$ of dimension $g$ that is not supersingular and has no $p$-torsion $k$-rational points \cite[Ch. 0.6]{li-oort98}.

\begin{prop}
	For every positive integer $\ell$, there exists a genus $g_\ell$ such that for all $g \geq g_\ell$ and $n \leq \ell$ the following holds: if $\rho(X) \in R_{g,n}$, then for the $p$-rank $f$ of $X$ we have that $f \leq n$.
\end{prop}

\begin{proof}
	By definition of supersingular elliptic curve, the supersingular part of a Poincar\'e isogeny decomposition of $X$ does not contribute to $X[p](k)$.
\end{proof} 

Except for a few cases, all values of the $p$-rank do occur. More precisely:

\begin{itemize}
	\item if $\rho(X) = 2g^2-g$, i.e.~$X$ is supersingular, then clearly $f(X) = 0$;
	
	\item if $\rho(X) \in R_{g,1}$, then $f(X) =1$ in this case, because an elliptic curve is supersingular if and only if its $p$-rank vanishes;
	
	\item if $\rho(X) \in R_{g,2}$, then $f(X) \in \lbrace 1,2 \rbrace$ and both cases do occur;
	
	\item if $\rho(X) \in R_{g,n}$ ($n \geq 3$), then $f(X) \in \lbrace 0, 1, \dots, n\rbrace$ and all cases do occur.
\end{itemize}

Here, we have used the following facts:

\begin{itemize}
	\item $\dim \ka_g = \frac{1}{2}g(g+1)$, $\ka_g$ being the moduli space of principally polarized abelian varieties\footnote{As is well-known, every $g$-dimensional abelian variety is isogenous to a principally polarized one.} of dimension $g$;
	\item $\dim \ks_{g,1} = \big[g^2/4\big]$ (integral part of $g^2/4$), where $\ks_{g,1}$ is the subscheme of $\ka_g$ consisting of supersingular abelian varieties;
	\item if $V_f$ is a connected component of the subscheme of $\ka_g$ consisting of abelian varieties with $p$-rank at most $f$, then $\dim V_f = \frac{1}{2}g(g+1) -g+f$.
\end{itemize}

We also would like to observe that the Newton polygon of an abelian variety has a very special structure when the Picard number is large. Indeed, we need only look at the Newton polygon $\Nwt_H$ of the $F$-crystal $H:=\rHcr{1}(X/W)$, as there are $F$-crystal isomorphisms $\bigwedge^n \rHcr{1}(X/W) \cong \rHcr{n}(X/W)$ for all $n$ (here $W = W(k)$ is the ring of Witt vectors). Moreover, the Newton polygon is invariant under isogenies, and thus we can use Corollary \ref{s_vs_rho} to list the Newton polygons of abelian varieties of large Picard number. Indeed, if $X$ has large Picard number $\rho(X) \in R_{g,n}$, we may assume $X=E^{g-n} \times A$, $E$ being a supersingular elliptic curve and $A$ an abelian variety with vanishing supersingularity index. Then, by the K\"unneth formula, 
\[ \rHcr{1}(X/W) \cong \rHcr{1}(E/W)^{\oplus(g-n)} \oplus \rHcr{1}(A/W),\]
and we can deduce the structure of $\Nwt_H$ from that of the $\Nwt_{\rHcr{1}(E/W)}$ and $\Nwt_{\rHcr{1}(A/W)}$. 

\section{Further remarks and comments}\label{pathologies}

\subsection{Dependence on the characteristic}

Now, we would like to compare Proposition \ref{self-product} to the corresponding result in characteristic zero \cite[Corollary 2.5]{hulek-laface17}. For Type IV, the bound in positive characteristic is sensibly weaker than the one over the complex numbers, and in fact it is the best possible. 

To see this, let us restrict ourselves to self-products of an abelian variety $A$ of dimension $n \geq 2$. In this situation, for Type IV, one has that $\rho(A^k) \leq nk^2$ when working over $\C$, while in characteristic $p >0$ the weaker bound $\rho(A^k) \leq (nk)^2$ holds and the value $\rho(A^k) = (nk)^2$ can be attained for abelian varieties of CM-type (cf. \cite[IV, \S{19}]{mumford08}).	

More concretely, in characteristic zero, $\rho(X) = g^2$ if and only if 
$X$ is isogenous to $E^g$, $E$ being an elliptic curve with complex multiplication \cite{katsura75}. On the other hand, in positive characteristic, we will now show that certain $g$-dimensional abelian varieties of CM-type have Picard number $\rho=g^2$. In fact, we can restrict ourselves to the case of a simple abelian variety.

Let $A$ be a $g$-dimensional simple abelian variety of CM-type, having Picard number $\rho(A) = g^2$. Then, it is straightforward to see that it necessarily has endomorphism algebra of Type $\rm{IV}(e_0,d)$ with $e_0 = 1$ and $d=g$. This case never occurs in characteristic zero thanks to the restrictions in Subsection \ref{restriction}, and by \cite[Remark 5.4]{oort88} the situation in positive characteristic depends on the characteristic of the base field. More precisely,

\begin{itemize}
	\item if $p$ is not-split in $D/\Q$ of Type $\rm{IV}(1,g)$, then $D$ does not occur as endomorphism algebra unless $g=1$;
	\item if $p$ is split in $D/\Q$ instead, then it can occur: in fact, for any choice of $0 < m < n$ with $(m,n)=1$ and $m+n=g$, there exists an abelian variety of dimension $g$ with endomorphism algebra of Type $\rm{IV}(1,g)$. 
\end{itemize}

This shows yet another pathology: given a positive integer $g$ and an Albert algebra $(D,*)$ satisfying the restrictions in Subsection \ref{restriction}, the existence of a simple abelian variety with endomorphism algebra isomorphic to $D$ depends on the properties of $p$ as an element of $D$.

\subsection{Non-additivity of the $R_g$'s}\label{6.2}

Let us go back to the computation of $R_4$. As $R_2 = \lbrace 1,2,3,4,6 \rbrace$, if we were in characteristic zero, we could immediately conclude that $12 \in R_{4}$ by \cite[Proposition 6.1]{hulek-laface17}. However, this is not the case in characteristic $p>0$: indeed, $\rho(A) =6$ if and only if $A$ is a supersingular abelian surface, and in order to produce $12 \in R_4$ one would naively take a product $A \times A'$ of two such surfaces. Unfortunately, these are always isogenous, hence $\rho(A \times A') = 28$. This shows yet another pathological behavior of the Picard number.

In fact, this phenomenon is not restricted to dimension four, but it occurs in every dimension $g \geq 4$. For $g \geq 7$, one has that $R_{g-2} + R_2 \nsubseteq R_g$: indeed, the largest Picard number in dimension $g-2$ is $2(g-2)^2 - (g-2)$, while it is $6$ in dimension two. Additivity would imply that $\big[2(g-2)^2 -(g-2)\big]+6$, which contradicts Theorem \ref{first_gaps}(2). In lower dimension, one can compute $R_g$ directly and see that the same argument applies, namely $21 \notin R_5$ (but $15 \in R_3$ and $6 \in R_2$) and $34 \notin R_6$ (whereas $28 \in R_4$ and $6 \in R_2$). In conclusion, this provides counterexamples to the additivity of the range of Picard numbers in positive characteristic in all dimensions $g \geq 4$.

\subsection{Yet another obstruction to addivity of the $R_g$'s}

We have just seen that one obstruction to an analogue of \cite[Proposition 6.1]{hulek-laface17} in positive characteristic is the presence of supersingular abelian varieties. However, at a closer look, this is not an crucial issue if one's aim is to describe the structure of $R_g$ recursively as in \cite[Section 6]{hulek-laface17}. We have already seen a hint to a possible fix in Subsection \ref{str_AV_large}: letting $R_g^*$ be the subset of $R_g$ consisting of those Picard numbers that can be attained by using abelian varieties with vanishing supersingularity index, one can formulate the following conjecture.

\begin{conj}
	The set of Picard numbers of $g$-dimensional abelian varieties can be described as follows:
	\begin{align*}	
	 R_g &= \bigcup_{k \vert g} \big\lbrace \rho(A^k) \, \vert \, \text{$A$ simple}, \ \dim A = g/k \big\rbrace \cup\\ 
	 & \bigcup_{1 \leq n \leq g-1} \big(R_n^* + R_{g-n}^* \big) \cup \bigcup_{1 \leq n \leq g-1} \big(\lbrace 2n^2-n \rbrace + R_{g-n}^* \big).
	 \end{align*}
\end{conj}

This makes sense as the integers in $R_g^*$ morally behave as the Picard numbers of complex abelian varieties. Also, notice that we have excluded the counterexamples of Subsection \ref{6.2}. In fact, the conjecture above would follow from:

\begin{conj}[Analogue of Proposition 6.3 in \cite{hulek-laface17} in positive characteristic]\label{countably_many}
	Let $g \geq 1$ and let $\rho \in R_g^{(p)}$ ($p>0$). Then, there exist countably many isogeny classes of abelian varieties of dimension $g$ and Picard number $\rho$.
\end{conj}

To prove Conjecture \ref{countably_many}, one can reduce to considering simple abelian varieties, and then carry on a Type-by-Type analysis. Unfortunately, here is where the theory of abelian varieties fails us: in fact, it is very hard to construct abelian varieties of given dimension and endomorphism algebra (satisfying the restrictions in Subsection \ref{restriction}). Moreover, it is not even known whether an endomorphism algebra of Type III or IV satisfying said restrictions always exists (this might also depend on $p$ as we have already seen earlier). Further results on this topic are surveyed in \cite{oort88} (in particular, \cite[\S 8]{oort88} contains a quick list of all known cases and open questions).

\subsection{Abelian varieties over finite fields}

By \cite[Corollary 8.2]{hulek-laface17}, it is known that, for every positive integer $g$, every realizable Picard number $\rho \in R_g$ of a complex abelian variety can in fact be realized over a number field. The argument therein use the technique of spreading out plus some considerations on Galois representations. One might hope for an analogue in positive characteristic, where one replaces numbers fields by finite field. However, the argument above does not carry over to positive characteristic.

In fact, the obstruction to such a result is deeply substantial and comes from the Tate conjecture. Let us fix a positive integer $g$ and let $X$ be an abelian variety defined over a finite field $k$. Then, the Tate conjecture implies that the 2\textsuperscript{nd} Betti number has the same parity of the Picard number, that is $b_2(X) \equiv \rho(X) \md 2$. This has been first noticed by Swinnerton-Dyer, as explained by Artin in \cite{artin74} (see also \cite{shioda81}).

Nevertheless, one may wonder whether all $\rho \in R_g$ such that $\rho \equiv b_2 \md 2$ occur as Picard numbers of abelian varieties over finite fields. This seems to be a rather subtle question, and we will now try to illustrate why. 

Let $\ell \geq 3$ be a positive integer, and let $g_\ell$ be as in Theorem \ref{distribution}. Fix $g \geq g_\ell$, and let $\rho \in R_{g,n}$ with $g \equiv n \md 2$ and $n \geq 3$. If $X$ has Picard number $\rho= 2(g-n)^2 - (g-n)+1$, then $X \sim E^{g-n} \times A_n$ as in the comment right after Corollary \ref{s_vs_rho} and $\rho(A_n)=1$. Although, $E$ is indeed defined over $\IF_{p^2}$, $A_n$ can never be defined over a finite field. Indeed, by a result of Tate \cite[Theorem 2]{tate66},
\[ 2g \leq \dim \End(X) \otimes \Q \leq (2g)^2.\]
As $\rho(A_n)=1$, $A_n$ must be simple and its endomorphism algebra is either $\Q$, a quaternion algebra over $\Q$ or a quadratic imaginary field, all of which contradict Tate's result. This implies that $E^{g-n} \times A_n$ is not defined over a finite field, but does not exclude that some other element in the isogeny class of $X$ might be.

\subsection{Noether-Lefschetz loci for abelian varieties ($\operatorname{char}$ $p \geq 0$)}

We would like to make a couple of remarks on (higher) Noether-Lefschetz loci for abelian varieties. In \cite[Remark 4.3]{hulek-laface17}, we had already notice that the moduli of complex abelian varieties does not behave as well as the moduli of K3 surface with respect to the Picard number. The reason for this was that the moduli information of an abelian variety $X/\C$ can be read off from a weight-one Hodge structure, while the information on the Picard number lives in degree two. A similar behavior occurs for abelian varieties in positive characteristic: this is a consequence of Theorem \ref{large_picard}. As a byproduct, we will now see that we do not have well-behaved (higher) Noether-Lefschetz loci ({NL loci} in short).

Let us fix a positive integer $g$ (which we take large enough, so that our consideration make sense, see Theorem \ref{large_picard}), and let $r \geq 1$ be another integer. One can naively define the following analogue of the Noether-Lefschetz loci: 

\[ \NL_g(r) := \big\lbrace [X] \in \ka_g \, \vert \, \rho(X) \geq r \big\rbrace.\]
Notice that, in any characteristic, the existence of gaps in the set of Picard numbers implies that many of these NL loci coincide. More precisely, if $r \notin R_g$ then $\NL_g(r)=\NL_g(r+1)$. Now we look at the situation for complex abelian varieties and some of the largest Picard numbers by making use of the results in \cite{hulek-laface17}. It is straightforward to see that $\NL_g(g^2)$ consists of singular abelian varieties, and $\dim \NL_g(g^2)=0$. Moving to the second largest Picard number, we have that 
\[\NL_g \big( (g-1)^2 +1 \big)= \NL_g(g^2) \cup \big\lbrace [X] \in \ka_g \, \vert \, X \sim E_\text{cm}^{g-1} \times F,\ E \not\sim F \big\rbrace,\]
where the subscript stands for \textit{complex multiplication}, and $\dim \NL_g \big( (g-1)^2 +1 \big) =1$. Next, 
\[\NL_g \big( (g-2)^2 +4 \big) = \NL_g \big( (g-1)^2 +1 \big) \cup \big\lbrace [X] \in \ka_g \, \vert \, X \sim E_\text{cm}^{g-2} \times F_\text{cm}^2,\ E \not\sim F \big\rbrace,\]
and also in this case $\dim  \NL_g \big( (g-2)^2 +4 \big) = 1$. Finally, we consider

\begin{align*}
	\NL_g \big( (g-2)^2 +3 \big) = \NL_g \big( (g-2)^2 +4 \big) &\cup \big\lbrace [X] \in \ka_g \, \vert \, X \sim E_\text{cm}^{g-2} \times F_\text{ord}^2\big\rbrace \\
	&\cup \big\lbrace [X] \in \ka_g \, \vert \, X \sim E_\text{cm}^{g-2} \times S_{\text{II}(1)}\big\rbrace,
\end{align*}
where $F_\text{ord}$ denotes an ordinary elliptic curve (that is $\End(E)\otimes \Q \cong \Q$) and $S_{\text{II}(1)}$ a simple abelian surface of Type II$(1)$; one can see that $\dim  \NL_g \big( (g-2)^2 +3 \big) = 1$ by results in \cite[Ch.~9]{birkenhake-lange04}. So we see that, although we are always adding new abelian varieties, the dimension of the NL loci might not increase upon lowering the bound on the Picard number.

However, if we are willing to bundle together certain Picard numbers, then we can get loci of $\ka_g$ that behaves quite nicely in this respect, at least if we work with large Picard numbers. Fix $\ell \in \IN$ and let $g_\ell$ be as in \cite[Theorem 7.4]{hulek-laface17}. Then, for $r \leq \ell$, we can define the following locus in $\ka_g$:
\begin{align*}
	\kl_g(r) :=& \big\lbrace [X] \in \ka_g \, \vert \, \rho(X) \in R_{g,n},\ \text{for some $n \leq r$} \big\rbrace.
\end{align*}
Equivalently, $\kl_g(r)$ is the locus in $\ka_g$ consisting of abelian varieties having Picard number $\rho \geq (g-r)^2 + 1$. By \cite[Corollary 7.6]{hulek-laface17}, it can alternatively be defined as the sublocus of $\ka_g$ of those abelian varieties containing a subvariety $Y$ isogenous to $E_\text{cm}^{g-r}$, $E_\text{cm}$ being again some elliptic curve with complex multiplication; in symbols: 
\[ \kl_g(r) = \big\lbrace [X] \in \ka_g \, \vert \, X \supset Y \sim E_\text{cm}^{g-r}\big\rbrace.\] 
We observe that $\kl_g(0) = \NL_g(g^2)$ and $\kl_g(1) = \NL_g \big( (g-1)^2+1 \big)$. Also, there is a chain of strict inclusions of subloci of $\ka_g$, namely
\[ \kl_g(0) \subsetneq \kl_g(1) \subsetneq \kl_g(2) \subsetneq \cdots, \subsetneq \kl_g(\ell -1 ) \subsetneq \kl_g(\ell).\] 
If $X \in \kl_g(r)$, then $X \sim E_\text{cm}^{g-r} \times A$, for some abelian variety $A$ of dimension $r$, and generically $\Hom(E_\text{cm},A)=0$. This means that $\kl_g(r)$ is a countable union of subvarieties of $\ka_g$ of dimension $\dim \ka_r$.

Turning to positive characteristic, the same approach works by virtue of Theorem \ref{large_picard}, Theorem \ref{distribution} and Corollary \ref{s_vs_rho}, and we can define the subloci $\kl_g(r)$ of $\ka_g$. However, counting dimensions yields that $\kl_g(r)$ is now a countable union of subvarieties of dimension $\dim \ks_{g-r,1} + \dim \ka_r$.

When $r > \ell$, the loci $\kl_g(r)$ do not behave as nicely, since the sets of the form $R_{g,r}$ may intersect each other and hence there are ambiguities on the structure of the corresponding abelian varieties up to isogenies.

\subsection{$R_g$ vs.~$R_g^*$}

Recall that in Section \ref{asymp_res}, we have introduced the subset $R_g^*$ of Picard numbers of abelian varieties of dimension $g$ and vanishing supersingularity index. For sake of completeness, we would like to report a few computations on $R_g$ and $R_g^*$, so that the reader can convince him- or herself that these sets differ more and more as $g \ra +\infty$. We have not proved this rigorously, but we do expect that $\#(R_g \setminus R_g^*) \ra +\infty$ as $g \ra +\infty$. In the following, the numbers in bold indicate Picard numbers that can be achieved with abelian varieties with nontrivial supersingularity index only (i.e. they constitute $R_g \setminus R_g^*$), and thus $R_g^*$ consists of the numbers not in bold.

\begin{align*}
	R_2 &= \lbrace 1,2,3,4,\bf{6} \rbrace \\
	R_3 &= \lbrace 1, \dots, 6, \bf{7}, \rm{9}, \bf{15} \rbrace \\
	R_4 &= \lbrace 1,\dots,7, \textbf{8}, \textbf{9}, 10, 16, \textbf{28} \rbrace \\
	R_5 &= \lbrace 1,\dots, 11, \textbf{12},15,\textbf{17},25,\textbf{29},\textbf{45} \rbrace \\
	R_6 &= \lbrace 1,\dots, 21, 22, \textbf{24},\textbf{26}, \textbf{29}, \textbf{\dots}, \textbf{32}, 36, \textbf{46}, \textbf{66} \rbrace
\end{align*}


\bibliographystyle{plain}
\bibliography{bib}{}

\hrulefill

\vspace{5mm}
\SMALL{
\noindent \textsc{Roberto Laface}
\\Technische Universit\"at M\"unchen
\\Zentrum Mathematik - M11 
\\85748 Garching bei M\"unchen (Germany)
\\ \textit{E-mail address: }{\tt laface@ma.tum.de}
}

\newpage

\end{document}

%% file: pack.tex
\usepackage{fancyhdr}

\usepackage{amsfonts}
\usepackage{amssymb}
\usepackage{amsthm}
\usepackage{amscd}
\usepackage{amsmath}
\usepackage{mathtools}
\usepackage{mathrsfs}
\usepackage{float}
\usepackage[all,cmtip]{xy}

\usepackage{color}	
\definecolor{due}{RGB}{0,76,147}

\usepackage{graphicx}	
\usepackage{multicol}	
\usepackage{wrapfig}
\usepackage[english]{babel} 
\usepackage[utf8]{inputenc}
\usepackage[T1]{fontenc}
\usepackage{enumitem}

\usepackage{longtable}
\usepackage{cite}

\theoremstyle{definition}
\newtheorem{defi}{Definition}[section]
\theoremstyle{plain}
\newtheorem{thm}[defi]{Theorem}

\newtheorem{prop}[defi]{Proposition}
\newtheorem{cor}[defi]{Corollary}
\newtheorem{lemma}[defi]{Lemma}
\theoremstyle{remark}

\newtheorem{conj}[defi]{Conjecture}

\newtheorem{ex}[defi]{Example}

\theoremstyle{definition}

\newtheorem*{ack}{Acknowledgement}

\newcommand{\ra}{\rightarrow}

\newcommand{\longra}{\longrightarrow}

\newcommand{\ka}{{\mathcal A}}

\newcommand{\kh}{{\mathcal H}}

\newcommand{\kl}{{\mathcal L}}

\newcommand{\ko}{{\mathcal O}}

\newcommand{\ks}{{\mathcal S}}

\newcommand{\C}{{\mathbb C}}

\newcommand{\IF}{{\mathbb F}}

\newcommand{\IH}{{\mathbb H}}

\newcommand{\IN}{{\mathbb N}}

\newcommand{\Q}{{\mathbb Q}}
\newcommand{\R}{{\mathbb R}}

\newcommand{\Z}{{\mathbb Z}}

\newcommand{\rH}{{\rm H}}
\newcommand{\rHet}[1]{{\rm H}_{\rm \acute{e}t}^{#1}}
\newcommand{\rHcr}[1]{{\rm H}_{\rm crys}^{#1}}
\newcommand{\Nwt}{{\rm Nwt}}

\newcommand{\NS}{\operatorname{NS}}

\newcommand{\rank}{\operatorname{rank}}

\newcommand{\End}{\operatorname{End}}
\newcommand{\Hom}{\operatorname{Hom}}

\newcommand{\NL}{\mathrm{NL}}
\newcommand{\md}{{\;}\text{mod}{\;}}

  \makeatletter
\newcommand{\xdashrightarrow}[2][]{\ext@arrow 0359\rightarrowfill@@{#1}{#2}}
\newcommand{\xdashleftarrow}[2][]{\ext@arrow 3095\leftarrowfill@@{#1}{#2}}
\newcommand{\xdashleftrightarrow}[2][]{\ext@arrow 3359\leftrightarrowfill@@{#1}{#2}}
\def\rightarrowfill@@{\arrowfill@@\relax\relbar\rightarrow}
\def\leftarrowfill@@{\arrowfill@@\leftarrow\relbar\relax}
\def\leftrightarrowfill@@{\arrowfill@@\leftarrow\relbar\rightarrow}
\def\arrowfill@@#1#2#3#4{%
  $\m@th\thickmuskip0mu\medmuskip\thickmuskip\thinmuskip\thickmuskip
   \relax#4#1
   \xleaders\hbox{$#4#2$}\hfill
   #3$%
}
\makeatother


%% file: picardnumbers2.bbl
\begin{thebibliography}{10}

\bibitem{albert34}
A.~Adrian Albert.
\newblock On the construction of {R}iemann matrices. {I}.
\newblock {\em Ann. of Math. (2)}, 35(1):1--28, 1934.

\bibitem{albert34a}
A.~Adrian Albert.
\newblock A solution of the principal problem in the theory of {R}iemann
  matrices.
\newblock {\em Ann. of Math. (2)}, 35(3):500--515, 1934.

\bibitem{albert35}
A.~Adrian Albert.
\newblock On the construction of {R}iemann matrices. {II}.
\newblock {\em Ann. of Math. (2)}, 36(2):376--394, 1935.

\bibitem{artin74}
M.~Artin.
\newblock Supersingular {$K3$} surfaces.
\newblock {\em Ann. Sci. \'Ecole Norm. Sup. (4)}, 7:543--567 (1975), 1974.

\bibitem{birkenhake-lange04}
C.~Birkenhake and H.~Lange.
\newblock {\em Complex abelian varieties}, volume 302 of {\em Grundlehren der
  Mathematischen Wissenschaften [Fundamental Principles of Mathematical
  Sciences]}.
\newblock Springer-Verlag, Berlin, second edition, 2004.

\bibitem{hulek-laface17}
K.~{Hulek} and R.~{Laface}.
\newblock {On the Picard numbers of abelian varieties}.
\newblock {\em to appear in Ann. Sc. Norm. Super. Pisa}.

\bibitem{husemoller04}
D.~Husem\"oller.
\newblock {\em Elliptic curves}, volume 111 of {\em Graduate Texts in
  Mathematics}.
\newblock Springer-Verlag, New York, second edition, 2004.
\newblock With appendices by Otto Forster, Ruth Lawrence and Stefan Theisen.

\bibitem{igusa60}
J.~Igusa.
\newblock Betti and {P}icard numbers of abstract algebraic surfaces.
\newblock {\em Proc. Nat. Acad. Sci. U.S.A.}, 46:724--726, 1960.

\bibitem{katsura75}
T.~Katsura.
\newblock On the structure of singular abelian varieties.
\newblock {\em Proc. Japan Acad.}, 51(4):224--228, 1975.

\bibitem{li-oort98}
K.-Z. Li and F.~Oort.
\newblock {\em Moduli of supersingular abelian varieties}, volume 1680 of {\em
  Lecture Notes in Mathematics}.
\newblock Springer-Verlag, Berlin, 1998.

\bibitem{mumford08}
D.~Mumford.
\newblock {\em Abelian varieties}, volume~5 of {\em Tata Institute of
  Fundamental Research Studies in Mathematics}.
\newblock Published for the Tata Institute of Fundamental Research, Bombay; by
  Hindustan Book Agency, New Delhi, 2008.
\newblock With appendices by C. P. Ramanujam and Yuri Manin, Corrected reprint
  of the second (1974) edition.

\bibitem{murty84}
V.~K. Murty.
\newblock Exceptional hodge classes on certain abelian varieties.
\newblock {\em Mathematische Annalen}, 268(2):197--206.

\bibitem{oort74}
F.~Oort.
\newblock Subvarieties of moduli spaces.
\newblock {\em Invent. Math.}, 24:95--119, 1974.

\bibitem{oort88}
F.~Oort.
\newblock Endomorphism algebras of abelian varieties.
\newblock In {\em Algebraic geometry and commutative algebra, {V}ol.\ {II}},
  pages 469--502. Kinokuniya, Tokyo, 1988.

\bibitem{shioda75}
T.~Shioda.
\newblock Algebraic cycles on certain {$K3$} surfaces in characteristic {$p$}.
\newblock In {\em Manifolds--{T}okyo 1973 ({P}roc. {I}nternat. {C}onf.,
  {T}okyo, 1973)}, pages 357--364. Univ. Tokyo Press, Tokyo, 1975.

\bibitem{shioda81}
T.~Shioda.
\newblock On the {P}icard number of a complex projective variety.
\newblock {\em Ann. Sci. \'Ecole Norm. Sup. (4)}, 14(3):303--321, 1981.

\bibitem{tate66}
J.~Tate.
\newblock Endomorphisms of abelian varieties over finite fields.
\newblock {\em Invent. Math.}, 2:134--144, 1966.

\end{thebibliography}
